\documentclass[10pt,leqno]{amsart}
\usepackage[ngerman,english]{babel} 
\usepackage[utf8]{inputenc} % Required for inputting international characters
\usepackage[T1]{fontenc} % Output font encoding for international characters
\usepackage{xcolor,colortbl}
\usepackage{indentfirst,csquotes}

\usepackage[top=0.5cm, textheight=20cm, textwidth=32cc,margin=2cm]{geometry}

\usepackage{paralist}

\usepackage{amssymb}
\usepackage{amsthm}
\usepackage{amscd}
\usepackage{amsfonts}
\usepackage{mathtools}
\usepackage{listings}
\usepackage{quiver}
\usepackage{adjustbox}
\usepackage{tabularx}
\usepackage{enumitem}
\usepackage{amsmath}

\usepackage{pgf,tikz}
\usepackage{tikz-cd}
\usetikzlibrary{matrix,arrows,calc}
\usepackage{pgfplots}
\pgfplotsset{compat=newest}
\usepackage{graphicx}
\usepackage{caption}
\usepackage{subcaption}
\usepackage{framed}
\usepackage{float}

\usepackage{hyperref}

\usepackage{thmtools}
\usepackage[capitalise]{cleveref}
\crefname{figure}{Figure}{Figures}
\crefname{equation}{}{}
\crefname{secinapp}{Appendix}{Appendices}
\Crefname{secinapp}{Appendix}{Appendices}

\usepackage[onehalfspacing]{setspace} %Zeilenabstand 1,5
\usetikzlibrary{decorations.markings}
\tikzset{
	halfarrow/.style={postaction={decorate},
		decoration={markings,mark=at position .5 with
			{\arrow{stealth}}}}}
\newcommand\defn[1]{\textbf{\color{black}#1}}
\newcommand{\R}{\mathbb{R}}
\newcommand{\Z}{\mathbb{Z}}
\newcommand{\A}{\mathcal{A}}
\newcommand{\defeq}{\coloneq}
\newcommand{\Dns}{\mathcal{D}_{n,s}}
\newcommand{\An}{\mathcal{A}_{n-1}}
\newcommand{\Bn}{\mathcal{B}_n}
\newcommand{\Dn}{\mathcal{D}_n}
\renewcommand{\epsilon}{\varepsilon}

\newcommand*{\abs}[1]{\left\lvert#1\right\rvert}
\newcommand{\precdot}{\prec\mathrel{\mkern-5mu}\mathrel{\cdot}}

\newtheorem{theorem}{Theorem}[section]

\newtheorem{prop}[theorem]{Proposition}

\newtheorem{lemm}[theorem]{Lemma}

\theoremstyle{definition}

\newtheorem{defi}[theorem]{Definition}
\newtheorem{bsp}[theorem]{Example}
\newtheorem{bem}[theorem]{Remark}
\newtheorem{nota}[theorem]{Notation}

\crefname{defi}{Definition}{Definitions}

\AtBeginDocument{}
\newsavebox{\imagebox}

\DeclareMathOperator{\rk}{rk}
\DeclareMathOperator{\ty}{typ}

\newcommand{\mA}{{\mathcal A}}
\newcommand{\mB}{{\mathcal B}}
\newcommand{\mD}{{\mathcal D}}
\newcommand{\mF}{{\mathcal F}}
\newcommand{\mI}{{\mathcal I}}
\newcommand{\mL}{{\mathcal L}}
\newcommand{\mR}{{\mathcal R}}
\newcommand{\mT}{{\mathcal T}}

\newcommand{\set}[1]{\left\{#1\right\}}

\usepackage[backend=bibtex,style=alphabetic, maxnames=99]{biblatex}
\hypersetup{ colorlinks=false, linkcolor=black, filecolor=black, urlcolor=black }

\subjclass[2020]{\noindent Primary 52C35 $\cdot$ Secondary 05C45 $\cdot$ 51F15 $\cdot$ 52C40}
\keywords{Hamiltonian cycle, hyperplane arrangement, simplicial arrangement, supersolvable\\\hspace*{3.7cm}arrangement, supersolvable oriented matroid}

\title{Hamiltonian Cycles in Simplicial and Supersolvable Hyperplane Arrangements} %%%%%%%%%%%%
\author[V. Körber]{Veronika Körber}
\author[T. Schnieders]{Tobias Schnieders}
\author[J. Stricker]{Jan Stricker}
\author[J. Walizadeh]{Jasmin Walizadeh}

\address{Institute for Mathematics, Goethe University Frankfurt, Robert-Mayer-Str. 10,\newline
	60325 Frankfurt am Main, Germany}
\address{Eberhard Karls Universität Tübingen,
	Geschwister-Scholl-Platz,\newline
	72074 Tübingen, Germany}
\address{Universität des Saarlandes, Campus,\newline66123 Saarbrücken, Germany\newline}

\email{veronika.koerber@uni-tuebingen.de}
\email{schnieders@math.uni-sb.de}
\email{stricker@math.uni-frankfurt.de}
\email{j.walizadeh@protonmail.com}

\begin{document}
	\begin{abstract}
		Motivated by the Gray code interpretation of Hamiltonian cycles in Cayley graphs,
		we investigate the existence of Hamiltonian cycles in tope graphs of hyperplane arrangements, with a focus on simplicial, reflection, and supersolvable arrangements.
		We prove that all supersolvable hyperplane arrangements have Hamiltonian cycles, offering a constructive proof based on their inductive structure. The proof can be extended to all supersolvable oriented matroids. 
		Furthermore, extending earlier results by Conway, Sloane, and Wilks, we prove that all restrictions of finite reflection arrangements, including all Weyl groupoids and crystallographic arrangements, admit Hamiltonian cycles. Finally, we confirm Hamiltonicity for all 3-dimensional simplicial
		arrangements listed in the Grünbaum--Cuntz catalogue, using that the infinite families
		listed are supersolvable arrangements.
		
	\end{abstract} 
	\maketitle

	\section{Introduction}
	There are many algorithms that construct all subsets or all permutations of a finite set successively. Such an algorithm is called a Gray code, if there is a way to successively derive one object from the last by a minimal change. Nijenhuis and Wilf describe a Gray code for the subsets of a finite set by adding or removing only one element in each step \cite{nijenhuis2014combinatorial}. Such Gray codes for permutations were given by Johnson \cite{johnson1963generation} and Trotter \cite{trotter1962algorithm} in the 1960s.
	
	We can interpret these Gray codes as Hamiltonian cycles in Cayley graphs of permutation groups generated by simple transpositions. It is widely conjectured that all finite Cayley graphs have a Hamiltonian cycle. Conway, Sloane, and Wilks found Hamiltonian cycles in the Cayley graphs of all finite reflection groups \cite{conway}.
	Additionally, Takato Inoue and Hiroyuki Yamane proved that the Cayley graphs of finite Weyl groupoids admit a Hamiltonian cycle \cite{yamane2021hamilton,inoue2023hamiltonian}. Reflection groups are groups that can be represented by hyperplane arrangements, where the group elements correspond to reflections along the hyperplanes. The graph of regions, also called tope graph, is precisely the Cayley graph for a certain set of generators.
	
	For each hyperplane arrangement, we can interpret the region graph as a tope graph, which is a subgraph of the cube graph. Therefore, for reflection groups, a Hamiltonian cycle on the Cayley graph is the same as a Hamiltonian cycle on the tope graph. In general, tope graphs of hyperplane arrangements do not have a Hamiltonian cycle.
	
	We consider supersolvable arrangements and oriented matroids and construct Hamiltonian cycles for all of them. While writing this paper, we learned that Sofia Brenner, Jean Cardinal, Thomas McConville, Arturo Merino, and Torsten Mütze \cite{brenner2025combinatorial} simultaneously and independently reached the same results for supersolvable hyperplane arrangements.
	
	In addition, we investigate the reflection arrangements and their restrictions, which are simplicial arrangements. These arrangements include the finite connected Weyl groupoids and crystallographic arrangements in dimension 3 and above. 
	We extend the proof of Conway, Sloane, and Wilks to all restrictions of reflection arrangements. These arrangements are also of particular interest, since it is conjectured that all inscribable zonotopes arise from these arrangements \cite{sanyal22}.
	
	Lastly, we consider the Grünbaum--Cuntz Catalogue of all known simplicial arrangements in dimension 3 \cite{CEL} and show that all their tope graphs have a Hamiltonian cycle. 
	
	\textbf{Acknowledgement.} This paper grew out of a project within a \emph{Dive into Research}, which grants research experience to undergraduates. We thank Professors Michael Cuntz, Lukas Kühne, and Raman Sanyal, who organised the Dive into Research, and we want to give special thanks to Professor Raman Sanyal, who posed the starting question. Jan Stricker was supported by the SPP 2458 \emph{Combinatorial Synergies}, funded by the 
	Deutsche Forschungsgemeinschaft (DFG, German Research Foundation) project ID 539866293. During the process of publication, Tobias Schnieders and Veronika Körber were supported by the DFG project ID 286237555. We are also thankful for the funding for the Dive into Research from the DFG priority programme SPP 2458 \emph{Combinatorial Synergies}.
	
	\section{Preliminaries}\label{section2}
	We start by defining the objects we are working with and state some of their basic properties. Here, for a given positive integer $m$, we always denote the set $\{1,2,\ldots,m\}$ as $[m]$. All computer calculations were performed using \texttt{SageMath} \cite{sagemath}.
	\begin{defi}
		A \defn{hyperplane} $H \subseteq \R^n$ is a codimension-1 affine subspace, which we can describe as
		\begin{align*}
			H\defeq\{x\in \R^n~ \mid ~\alpha_1x_1+\alpha_2x_2+\cdots+\alpha_nx_n=b\}
		\end{align*}
		for some $\alpha=(\alpha_1,\ldots,\alpha_n)  \in \R^n\setminus\{0\}$ and some $b\in \R$. The vector $\alpha$ is called a \defn{normal} of $H$. We denote a hyperplane defined by $\alpha$ and $b=0$ as $H_{\alpha}$.
	\end{defi}
	
	The choice of a normal vector induces an orientation of the hyperplane.
	\begin{nota}
		We denote the closed/open half spaces of a hyperplane $H_{\alpha}$ as
		\begin{align*}
			H_{\alpha}^{\geq} &\defeq\{x\in \R^n \mid \alpha^tx\geq 0\},\\
			H_{\alpha}^{\leq} &\defeq\{x\in \R^n \mid \alpha^tx\leq 0\},\\
			H_{\alpha}^{>} &\defeq\{x\in \R^n \mid \alpha^tx>0\},\\
			H_{\alpha}^{<} &\defeq\{x\in \R^n \mid \alpha^tx<0\}.\\
		\end{align*}
	\end{nota}
	\begin{defi}
		A \defn{hyperplane arrangement} $\mA$ is a finite non-empty set of hyperplanes. If all hyperplanes in $\mA$ contain the origin, we call $\mA$ a \defn{central} hyperplane arrangement.
	\end{defi}
	
	In the rest of the paper we focus only on central hyperplane arrangements.
	
	\begin{defi}
		The \defn{intersection lattice} of a central hyperplane arrangement $\mA = \set{H_1,\ldots,H_m}$ in $\R^n$ is the set 
		\[
		\mL(\mA) = \set{\bigcup_{i \in I} H_i \mid I \subset [m]},
		\]
		partially ordered by reverse inclusion. This is a geometric lattice. The \defn{rank} of an element $X \in \mL(\mA)$ is its codimension in $\R^n$. The \defn{rank} of a hyperplane arrangement $\mA$ is the codimension of $\bigcap_{H \in \mA} H$. If the rank equals the dimension of the ambient space, we call $\mA$ \defn{essential}. For an essential hyperplane arrangement, the \defn{atoms} of the lattice are the hyperplanes, the \defn{coatoms} are the rank-$(n-1)$ elements.
	\end{defi}
	\begin{defi}
		The complement of the arrangement in the ambient space $\R^n\setminus \bigcup_{H\in\mA}H$ is disconnected. The closures of the connected components are called \defn{regions}. We denote the set of regions as $\mR(\mA)$.
		The intersection of two regions with affine codimension $1$ is a \defn{wall} of both regions. Each wall is inside a hyperplane of the arrangement. A region is called \defn{simplicial} if the normal vectors of the hyperplanes of all its walls are linearly independent. If all regions of $\mA$ are simplicial, we call $\mA$ a \defn{simplicial} hyperplane arrangement.
	\end{defi}
	\begin{defi}
		Let $\mA$ be a hyperplane arrangement. By the \defn{deletion} of a hyperplane $H$, we obtain the hyperplane arrangement $\mA \setminus \set{H}$ in $\mathbb{R}^n$. We call an arrangement that results by multiple deletions from $\mA$, a \defn{subarrangement} of $\mA$.
		
		The \defn{restriction} of $\mA$ to a hyperplane $H$ results in a hyperplane arrangement
		\[\mA^H = \set{H' \cap H \mid H \neq H' \in \mA}.\]
	\end{defi}
	
	When passing to a subarrangement, regions merge or stay inert. Hence, we can define the following surjective map of regions.
	
	\begin{defi}\label{defi-surjSubArr}
		Let $\mA'$ be a subarrangement of $\mA$.
		We define a map
		\begin{align*}
			&\pi\colon \mR(\mA) \rightarrow \mR(\mA'),\\
			&\text{such that } \pi(R) \text{ is the unique region of }\mA' \text{ containing } R.
		\end{align*}
	\end{defi}
	\begin{bem}
		Clearly, the map presented in \cref{defi-surjSubArr} is well-defined and surjective.
	\end{bem}
	
	\begin{defi}
		Let $\mA=\{H_1,\ldots,H_m\}$ be a central hyperplane arrangement with a fixed order and orientation of the hyperplanes. A \defn{tope} $s_R$ of a region $R$ is a map $s_R \colon [m] \rightarrow \{-, +\}$, where $s_R(k) = +$ if $R \subseteq H_k^{\geq}$ and  $s_R(k) = -$ if $R \subseteq H_k^{\leq}$. We often identify $s_R$ with the vector $(s_R(1),s_R(2),\ldots,s_R(m)) \in \set{-,+}^m$.
		
		The collection of these functions $s \colon [m] \rightarrow \{-, +\}$ forms the vertex set of the \defn{tope graph} for a hyperplane arrangement, with two such functions (vertices) being adjacent if their values differ in exactly one position. We note the tope graph of $\mA$ as $\mT(\mA)$.%Wir nennen diese Kanonische Beziehung vorher und die Abbildungen bilden ja einzelne indexe auf -,+ ab
	\end{defi}
	
	For our cases, it is helpful to define the type of an edge in tope graphs of hyperplane arrangements.
	
	\begin{defi}
		Let $\mA$ be a hyperplane arrangement. The \defn{type} of an edge $e$ between two regions $R_1$ and $R_2$ is the hyperplane $H\in \mA$, which contains the separating wall between $R_1$ and $R_2$. We write this as $\ty(e)=H$.
	\end{defi}
	
	We describe how the tope graph behaves under the deletion of hyperplanes from an arrangement.
	
	\begin{lemm}\label{contracHam}
		Let $\mA$ be a hyperplane arrangement and $\mA'=\mA\setminus\{H_1,\ldots, H_k\}$ a subarrangement. Let $E^-=\{e\in E(\mT(\mA))\mid\ty(e)\in \{H_1,\ldots, H_k\}\}$ be the set of edges of a type in the set of deleted hyperplanes. The tope graph of $\mA'$ is the graph that we obtain by contracting all edges in $E^-$ in $\mT(\mA),$ identifying parallel edges and deleting loops.
	\end{lemm}
	
	\begin{defi}
		Let $\mA_1$ in $\R^n$ and $\mA_2$ in $\R^{n'}$ be hyperplane arrangements. We define the \defn{product} $\mA_1 \times \mA_2$ as\[\mA_1 \times \mA_2\defeq \{H\oplus \R^{n'}\mid H\in \mA_1\}\cup \{\R^n\oplus H\mid H\in \mA_2\}.\]
		If an arrangement $\mA$ can be written as the product of two non-empty arrangements, it is called \defn{reducible}. Otherwise, we call $\mA$ \defn{irreducible}.
	\end{defi}
	\begin{bem}
		The tope graph of $\mA_1 \times \mA_2$ is the product of the tope graphs of $\mA_1$ and $\mA_2$.
	\end{bem}
	
	The following \cref{prodHam} justifies focusing only on Hamiltonian cycles in irreducible hyperplane arrangements.
	\begin{prop}\label{prodHam}
		If a hyperplane arrangement $\mA$ is the product of two hyperplane arrangements $\A_1$ and $\A_2$ that have a Hamiltonian cycle, then $\mA$ has a Hamiltonian cycle.
	\end{prop}
	\begin{proof}
		Suppose the hyperplane arrangement $\mA$ is the product of two hyperplane arrangements $\A_1$ and $\A_2$ with Hamiltonian cycles $\{a_0, a_1, \ldots , a_r \}$ and $\{b_0, b_1, \ldots , b_s \}$, respectively.
		Because the tope graph of $\mA$ is the product of the tope graphs of $\mA_1$ and $\mA_2$, the cycle
		\begin{align*}
			\{&(a_0, b_0), (a_0,b_1), \ldots,(a_0,b_s),\\
			&(a_1, b_s), (a_1,b_{s-1}), \ldots,(a_1,b_0),\\
			&(a_2, b_0), (a_2,b_1), \ldots,(a_2,b_s),\\
			&(a_3, b_s), (a_3,b_{s-1}), \ldots, \\
			&\ldots, (a_r,b_0)\}
		\end{align*}
		(reading from left to right) is a Hamiltonian cycle for $\A_1 \times \A_2$.
	\end{proof}
	
	\section{Supersolvable oriented matroids}\label{section5}
	
	Supersolvable hyperplane arrangements appear regularly in the research of hyperplane arrangements \cite{mu2015supersolvability,cuntz2019supersolvable}. Supersolvable arrangements include the reflection arrangements of type $\An$, $\Bn$ and $\mI_m$ \cite{hoge2014supersolvable}.
	
	\begin{defi}
		We call an element $v$ of a ranked finite lattice $L$ \defn{modular} if\[\rk(v\vee w)+\rk(v\wedge w)=\rk(v)+\rk(w)\]
		for all $w\in L$. 
		A finite lattice $L$ is \defn{supersolvable} if it is ranked and there exists a maximal chain 
		\[\hat{0}=v_0\precdot v_1\precdot v_2 \precdot \ldots \precdot v_{n-1} \precdot v_n=\hat{1}\] of modular elements.
	\end{defi}
	\begin{defi}
		We call a hyperplane arrangement \defn{supersolvable} if its intersection lattice is
		supersolvable. We call an oriented matroid \defn{supersolvable} if its lattice of flats is supersolvable.
	\end{defi}
	\begin{theorem}\label{thm:supersolvableishamiltonian}
		The tope graphs of supersolvable oriented matroids and supersolvable hyperplane arrangements admit a Hamiltonian cycle.
	\end{theorem}
	
	To make the proof more accessible for the reader, we state the proof in terms of hyperplane arrangements and their regions. All oriented matroids can be represented as pseudosphere arrangements by \cite{FOLKMAN1978199}. These pseudospheres and their regions behave analogously to hyperplanes and their regions for supersolvable arrangements. All properties that we need for the proof are true for both supersolvable hyperplane arrangements and supersolvable oriented matroids \cite[Section 6]{BjEdZi}. This bridges the gap between oriented matroids and hyperplane arrangements and makes the proof work for oriented matroids analogously.
	
	Supersolvable arrangements have a recursive structure, which are important for the construction of a Hamiltonian cycle.
	\begin{theorem}[{\cite[Theorem 4.3.]{BjEdZi}}]\label{sszerlegung}
		All rank-2 hyperplane arrangements are supersolvable. 
		An arrangement $\mA$ of rank $n \geq 3$ is supersolvable if and only if it can be written as a disjoint union of arrangements $\mA=\mA_0\uplus\mA_1$ (with $\mA_1\neq\varnothing$), where $\mA_0$ is a supersolvable arrangement of rank $n-1$, and for any $H', H'' \in \mA_1$, there is an $H\in\mA_0$ such that $H'\cap H''\subseteq H$.
	\end{theorem}
	
	Recall the surjective map $\pi\colon\mR(\mA)\rightarrow\mR(\mA_0)$ on regions of (sub)arrangements from \cref{defi-surjSubArr}.
	
	\begin{defi}\label{deffiber}
		Let $\mA$ be a hyperplane arrangement with partition $\mA=\mA_0\uplus\mA_1$.
		We define the \defn{fiber} of a region $R\in \mR(\mA)$ as $\mF(R)\defeq\pi^{-1}(\pi(R))\subseteq \mR(\mA)$.
	\end{defi}
	Given a region $R$, the fiber $\mF(R)$ contains all regions that cover the same area as one region of $\mA_0$. By definition, the regions in $\mF(R)$ are separated only by hyperplanes in $\mA_1$. There is a one to one correspondence between the regions of $\mA_0$ and the fibers of $\mA$.
	
	\begin{figure}[htp]
		\centering
		\begin{tikzpicture}[dot/.style={circle,inner sep=1pt,fill,label={#1},name=#1},
			extended line/.style={shorten >=-#1,shorten <=-#1},
			extended line/.default=4cm,scale=.6]
			\coordinate (A) at (0,0);
			\coordinate (B) at (1,0);
			\coordinate (C) at (1.31,0.95);
			\coordinate (D) at (.5,1.54);
			\coordinate (E) at (-.31,0.95);
			\begin{scope}
				\clip (-3,-0.2) rectangle (6,4);
				\fill[orange,opacity=0.3] (0.5,0.688) -- ($(0.5,0.688)!8!(C)$) -- ($(0.5,0.688)!-8!(A)$) -- cycle;
				\draw[blue,extended line] (A) -- ($(C)!.5!(D)$);
				\draw[blue,extended line] (B) -- ($(D)!.5!(E)$);
				\draw[blue,extended line] (C) -- ($(E)!.5!(A)$);
				\draw[blue,extended line] (D) -- ($(A)!.5!(B)$);
				\draw[blue,extended line] (E) -- ($(B)!.5!(C)$);
				\draw[extended line] (A) -- (B);
				\draw[extended line] (B) -- (C);
				\draw[extended line] (C) -- (D);
				\draw[extended line] (D) -- (E);
				\draw[extended line] (E) -- (A);
				
				\node at (2.5,2) {$\mathcal{R}$};
			\end{scope}
		\end{tikzpicture}
		\caption{The blue hyperplanes are the hyperplanes of $\mA_0$. The shaded area is the fiber $\mF(R)$ of a region $R,$ see \cref{deffiber}.}
		\label{exampleFIber}
	\end{figure}
	From now on, for a supersolvable hyperplane arrangement $\mA$, let $\mA=\mA_0\uplus\mA_1$ always be the partition from \cref{sszerlegung}.
	
	\begin{theorem}[{\cite[Proof of Theorem 4.6.]{BjEdZi}}]\label{ssfiberh1}
		Let $\mA$ be a supersolvable hyperplane arrangement. The induced subgraph of every fiber is a path of length $\abs{\mA_1}$ and each hyperplane of $\mA_1$ occurs in the path as the type of edges exactly once.
	\end{theorem}
	
	\begin{defi}\label{canonicalRegion}
		A \defn{canonical base region} $B_0$ of $\mA$ is defined recursively using $\pi$. In an arrangement of rank 2, any region is canonical. For $\mA$ of rank $n>2$, a region $B_0$ is canonical if $\pi(B_0)$ is canonical and $B_0$ is an endpoint of the path from $\mF(B_0)$.
	\end{defi}
	
	From now on, let $B_0$ always be a canonical base region and all hyperplanes oriented such that $B_0$ is on the positive side of the hyperplanes.
	\begin{lemm}\label{epsLemm}
		Let $\mA$ be a supersolvable hyperplane arrangement. Then, for a region $B\in\mR(\mA_0)$, there exists vertices $\epsilon_+(B),\epsilon_-(B)\in \pi^{-1}(B)$ whose topes for $\mA_1$ have only positive signs or negative signs, respectively.
		
		In addition, the regions $\epsilon_+(B)$ and $\epsilon_-(B)$ are the endpoints of the path from $B$ and have only one neighbouring region in $\pi^{-1}(B)$. All other neighbours are $\epsilon_+(B')$ or $\epsilon_-(B')$, respectively, for neighbouring regions $B'\in\mR(\mA_0)$ to $B$.
	\end{lemm}
	\begin{proof}
		By \cref{sszerlegung} and the definition of an endpoint in a fiber, endpoints only have one neighbouring region inside the fiber and edges to all other incident fibers.
		
		To prove that $\epsilon_+(B)$ and $\epsilon_-(B)$ are well-defined and the endpoints of the path of $B$, it suffices to show that $\epsilon_+(B)$ is well-defined and an endpoint, since the endpoints have flipped topes for $\mA_1$.
		
		By \cref{canonicalRegion}, $B_0$ is an endpoint in $\mF(B_0)$ and $B_0=\epsilon_+(\pi(B_0))$. Let $B'$ be a neighbouring region to $B_0$ via an edge of type $H\in\mA_0$. Since we only switch a sign for one hyperplane in $\mA_0$, we have $B'=\epsilon_+(\pi(B'))$.
		
		Assume $B'$ is not an endpoint of the path in $\mF(B')$. Let $B_1'$ and $B_2'$ be the endpoints of $\mF(B')$. Let $S_i'$ be the set of hyperplanes in $\mA_1$ that have different signs in the tope of $B_i'$ for $i\in\{1,2\}$ compared to $B'$. By \cref{ssfiberh1}, we know that $S_1'\uplus S_2'=\mA_1$. Since $B_i'$ is an endpoint, it has a neighbour $B_i$ in $\mF(B_0)$ via an edge of type $H$. Let $S_i$ be the set of hyperplanes in $\mA_1$ that have different signs in the tope of $B_i$ compared to $B_0$. By construction, it holds that $S_i=S_i'$. Since $B_1$ and $B_2$ are on the path in $\mF(B_0)$, either $S_1\subseteq S_2$ or $S_2\subseteq S_1$ holds. This is a contradiction, since $S_1\cap S_2=S_1'\cap S_2'=\emptyset$.
		
		By induction on the distance to $B_0$, because $\mT(\mA_0)$ is connected, all end points reached by $B_0$ this way are the vertices $\epsilon_+(B)$ for all regions $B\in\mR(\mA_0)$. It follows, that for each fiber, $\epsilon_+(B)$ is well-defined and an endpoint.
		
		At last, outside of a fiber, the edges are only of types in $\mA_0$. The topes of $\mA_1$ do not change and $\epsilon_+$ regions have only other $\epsilon_+$ regions as neighbours, and the same is true for $\epsilon_-$ regions.
	\end{proof}
	
	\begin{proof}[Proof of \cref{thm:supersolvableishamiltonian}]
		We mainly argue about the sign patterns of topes.
		The argument then carries over verbatim from supersolvable arrangements to supersolvable oriented matroids.
		
		Let $\mA$ be a supersolvable hyperplane arrangement. We prove the statement by induction on $n = \rk(\mA)$. For $n=2$, all hyperplane arrangements trivially have a Hamiltonian cycle since their tope graphs are always just a cycle.
		
		Let $n>2$ and assume all supersolvable hyperplane arrangements with rank smaller than $n$ admit a Hamiltonian cycle in their tope graph.
		
		According to \cref{sszerlegung}, the subarrangement $\mA_0$ is supersolvable and of rank $(n-1)$.
		
		By the induction hypothesis, the tope graph of $\mA_0$ has a Hamiltonian cycle. We write $B_1,B_2,\ldots,B_{2k},B_1$ for the Hamiltonian cycle in $\mA_0$. Because of \cref{ssfiberh1}, we can traverse the fiber of each $B_i$ by a path meeting all regions in the fiber.
		
		Let $P_+(B_i)$ be the path of $B_i$ starting in $\epsilon_+(B_{i})$ and ending in $\epsilon_-(B_{i})$. Let $P_-(B_i)$ be defined accordingly.
		
		Consider a part of the Hamiltonian cycle $B_{i-1},B_i,B_{i+1}$. By Lemma \cref{epsLemm}, the regions $\epsilon_+(B_{i-1})$ and $\epsilon_-(B_{i-1})$ are neighbours of $\epsilon_+(B_{i})$ and $\epsilon_-(B_{i})$, respectively, since $B_{i-1}$ and $B_i$ are adjacent in the tope graph $\mT(\mA_0)$. The same holds for $\epsilon_+(B_{i+1})$ and $\epsilon_-(B_{i+1})$. W.l.o.g.,\ assume we started the path of $B_{i-1}$ in $\epsilon_+(B_{i-1})$. We traverse the path $P_+(B_{i-1})$ and end in $\epsilon_-(B_{i-1})$. There is an edge between $\epsilon_-(B_{i-1})$ and $\epsilon_-(B_{i})$. We can connect $P_+(B_{i-1})$ to $P_-(B_i)$ via this edge and traverse the path $P_-(B_i)$ and end in $\epsilon_+(B_{i})$, where we can go to $\epsilon_+(B_{i+1})$ in the next path $P_+(B_{i+1})$.
		In \cref{fig:sphereExampleSuper}, there is an example of how two paths in adjacent fibers traverse the region graph of a supersolvable arrangement.
		
		We can traverse all fibers by $P_+(B_1),P_-(B_2),P_+(B_3),\ldots,P_-(B_{2k})$ one after the other, because $\mA_0$ has an even number of regions. The path $P_-(B_{2k})$ ends in $\epsilon_+(B_{2k})$ and since $B_{2k}$ and $B_1$ are adjacent in $\mT(\mA_0)$, there exists an edge $(\epsilon_+(B_{2k}),\epsilon_+(B_{1}))$ that closes the Hamiltonian cycle for $\mT(\mA)$.
	\end{proof}
	
	\begin{figure}[htp]
		\centering
		\makebox[0.5\linewidth]{
			\includegraphics[page=1,width=0.45\textwidth]{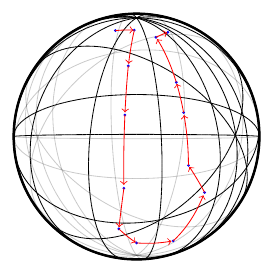}
		}
		\caption{An example of a part of a Hamiltonian cycle of a 3-dimensional arrangement, intersected with a sphere, traversing two neighbouring fibers.}
		\label{fig:sphereExampleSuper}
	\end{figure}
	%\pagebreak
	\newpage
	\section{Restrictions of reflections arrangements}\label{section6}
	An important class of simplicial hyperplane arrangements is the class of finite reflection arrangements. In \cite{conway}, Conway, Sloane, and Wilks showed that their tope graphs have Hamiltonian cycles. Furthermore, Takato Inoue and Hiroyuki Yamane proved that the Cayley graphs of finite Weyl groupoids admit a Hamiltonian cycle \cite{yamane2021hamilton,inoue2023hamiltonian}.
	
	\begin{theorem}\label{allResRefl}
		All restrictions of reflection arrangements admit a Hamiltonian cycle.
	\end{theorem}
	
	There are only finitely many reflection arrangements for a given rank. These arrangements are characterised by the different Coxeter types, see \cite{coxeter1934discrete,grove1996finite}. The irreducible Coxeter types are $\mA_n$ ($n\geq 1$), $\mB_n$ ($n\geq 2$), $\mD_n$ ($n\geq 4$), $E_6$, $E_7$, $E_8$, $F_4$, $H_3$, $H_4$, and $\mI_m$ ($m=5$ or $m>7$), see \cite[Theorem 9]{coxeter1934discrete}. Using \cref{thm:supersolvableishamiltonian}, we can already construct Hamiltonian cycles for the reflection arrangements of type $\mA_n$, $\mB_n$ and $\mI_m$.
	
	Restrictions of simplicial arrangements are again simplicial arrangements. From \cite[Section 6.5]{orlik2013arrangements}  (as cited by \cite{sanyal22}), we know the relations of all reflection arrangements to their restrictions. In \cref{fig:hasse-diagram}, we visualise this by a Hasse diagram, where the lines represent which hyperplane arrangements we obtain when restricting to certain hyperplanes. The restrictions of reflection arrangements include all Weyl groupoids and crystallographic arrangements, see \cite{CuntzHeckenberger+2015+77+108}.
	
	%\begin{figure}[tp]
	%	\centering
	%	\makebox[\linewidth]{
		%		\includegraphics[page=1,width=\textwidth]{./hasse-diagram.pdf}
		%	}
	%	\caption{A Hasse diagram showing how restrictions of reflection arrangements relate to each other. Here, $\boxed{m,k}$ refers to the $k$-th arrangement with $m$ hyperplanes, as listed in \cite{CEL}.}
	%	\label{fig:hasse-diagram}
	%\end{figure}
	In \cref{fig:hasse-diagram}, we can see that, except for the arrangements of type $\Dns$, there are only finitely many restrictions of reflection arrangements that are not combinatorially isomorphic to reflection arrangements. We call these the \defn{sporadic restrictions}.
	\begin{theorem}\label{sporadicRest}
		All sporadic restrictions admit a Hamiltonian cycle.
	\end{theorem}
	\begin{proof}
		Since there are only finitely many sporadic restrictions, we constructed an algorithm that efficiently computes the tope graph of simplicial arrangements and finds Hamiltonian cycles for them. We explain the algorithm in Appendix \labelcref{appendix:algo}. The graphs and Hamiltonian cycles of these restrictions can be found in Appendix \labelcref{appendix:sporadic}.\footnote{The algorithm, as well as the graphs, can be found on \cite{githubdb}.}
	\end{proof}
	%\tobias{Frage: Dieser Block könnte stattdessen in die Preliminaries. Soll er trotzdem hier bleiben?}
	The only restrictions of hyperplane arrangements that are not sporadic arrangements, are the arrangements of type $\Dns$. They are closely related to the arrangements of type $\A_{n-1}$ and $\mB_n$.
	Let $e_i$ be the $i$-th standard basis vector in $\R^n$, and recall that $H_{v} = \set{x \in \R^n \mid v^t x = 0}$ for $v \in \R^n\setminus\{0\}$.
	
	\begin{defi}
		We define the reflection arrangements of type $\mA_{n-1}$ in $\R^n$ as \[\A_{n-1} = \set{H_{e_i-e_j}\mid 1 \leq i < j \leq n}.\]
	\end{defi}
	\begin{bem}\label{bemAn}
		The $n!$ regions of $\A_{n-1},$ and thus also the vertices of its tope graph, are in bijection to the permutations in $\mathfrak{S}_n.$ %\veronika{Vorschlag: Ergänzende Erklärung}. %We identify a region with its permutation.	
		The $n!$ regions of $\A_{n-1}$ are in bijection to the permutations in $\mathfrak{S}_n$ in the following way. A hyperplane of type $H_{e_i-e_j}$ separates all $x\in \R^n\setminus H_{e_i-e_j}$ depending on which of the entries $x_i$ or $x_j$ is larger. Thus, every region is determined by a specific order of these entries. If we identify the region, where for a $x$ in its interior it holds that $x_1>x_2>...>x_n$ with the identity permutation. Then the permutation $\sigma$ is always identified with the region where $x_{\sigma(1)}>x_{\sigma(2)}>...>x_{\sigma(n)}.$
		
		Two vertices with corresponding permutations $\sigma,\sigma'\in \mathfrak{S}_n$ in $\mT(\mA_{n-1})$ are adjacent via an edge of type $H_{e_i-e_j}$ for some $i,j\in[n]$ if the permutations only differ by the simple transposition that swaps $i$ and $j$ because the order of the components of points in two regions that share a wall of type $H_{e_i-e_j}$ only differ in the entries of $\sigma(i)$ and $\sigma(j).$
	\end{bem}
	
	\begin{figure}[tp]
		\centering
		\makebox[\linewidth]{
			\includegraphics[page=1,width=\textwidth]{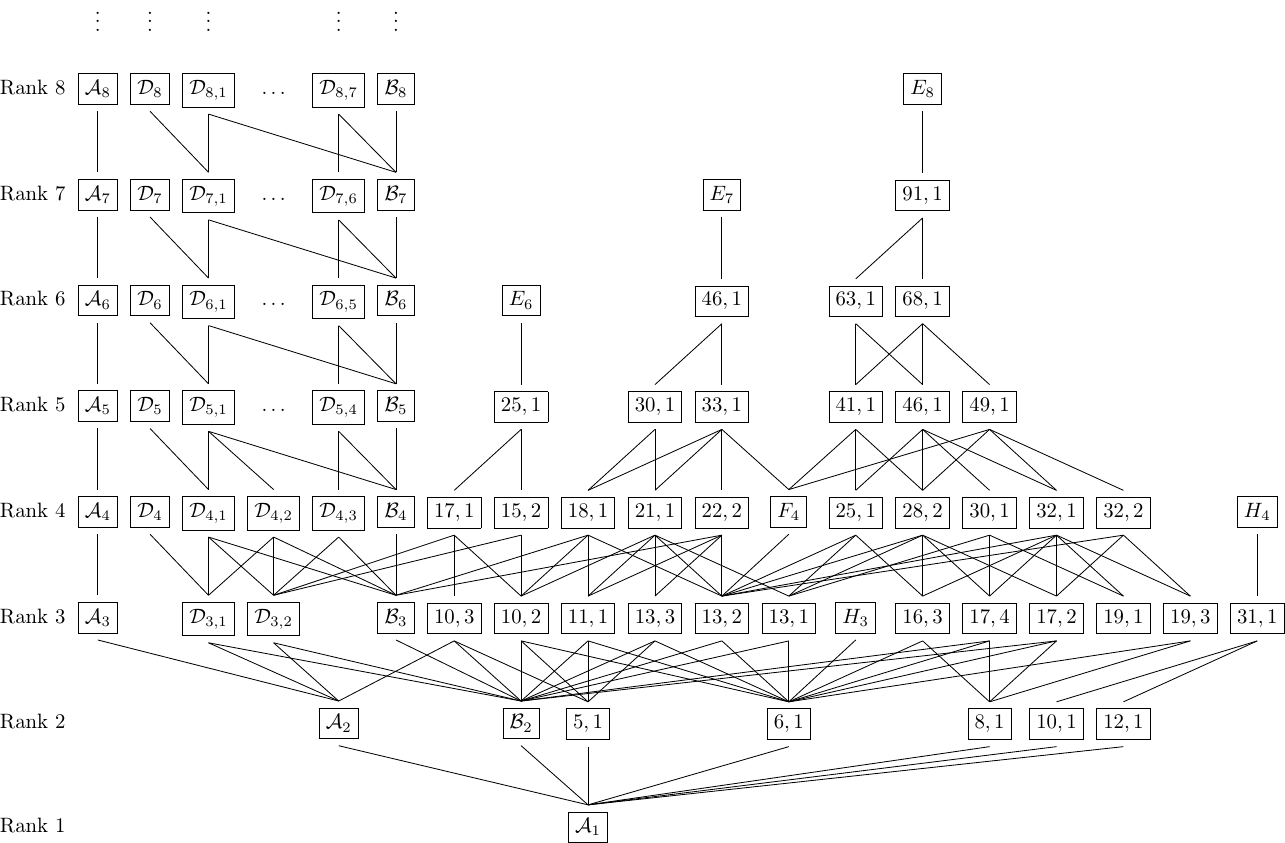}
		}
		\caption{A Hasse diagram showing how restrictions of reflection arrangements relate to each other. Here, $\boxed{m,k}$ refers to the $k$-th arrangement with $m$ hyperplanes, as listed in \cite{CEL}.}
		\label{fig:hasse-diagram}
	\end{figure}
	
	\begin{defi}
		We define the reflection arrangement of type $\Bn$ in $\R^n$ as\[\Bn = \set{H_{e_i\pm e_j} \mid 1 \leq i < j \leq n} \cup \set{H_{e_i} \mid 1 \leq i \leq n}.\]
	\end{defi}
	\begin{bem}\label{bemBn}
		The regions of $\mB_n$ are in bijection to the signed permutations $V\defeq \mathfrak{S}_n\rtimes \Z_2^n$, where every component of $\Z_2^n$ acts non-trivially on $\mathfrak{S}_n$, in the following way. %\veronika{Vorschlag: ergänzende Erklärung}
		We start as in \cref{bemAn} and note that now additionally, the hyperplanes  $H_{e_i+e_j}$ also sort all $x\in\R^n\setminus H_{e_i+e_j}$ by the absolute values of their entries and the hyperplanes $H_{e_i}$ also determine their signs. %The region where for a $x$ in its interior all entries are positive and $x_1>x_2>...>x_n$ is identified with the identity permutation.
		%Then,
		A signed permutation $(\sigma,\delta)$ is identified with the region where $|x_{\sigma(1)}|>|x_{\sigma(2)}|>...>|x_{\sigma(n)}|$ and $x_{\sigma(i)}>0$ exactly if $\delta_i=1.$
		
		%We identify a region of $\mB_n$ with its signed permutation.
		
		Two vertices with corresponding signed permutations $(\sigma,\delta)$ and $(\sigma',\delta')$ are adjacent via an edge of type $H_{e_k}$ for some $k\in[n]$, if $\sigma=\sigma'$, $\delta(k)=-\delta'(k)$, $\delta(l)=\delta'(l)$ for all $l\ne k$, and $\sigma(n)=k$. This is because for a point in the interior of a region, as in the coordinate $\sigma(n)$ there is the entry with the smallest absolute value, the only wall of the form $H_{e_i}$ can be $H_{e_{\sigma(n)}}$. %For the other hyperplanes $H_{e_i},$ other hyperplanes are in between and thus they are not walls of this region.
		Moreover, two such vertices are incident via an edge of type $H_{e_i\pm e_j}$ for some $i,j\in[n]$ if $\delta=\delta'$ and $\sigma$ and $\sigma'$ only differ by the simple transposition switching $i$ and $j,$ analogously as described in \cref{bemAn}.
		
		For a fixed $\delta'\in\Z_2^n$, we define
		\begin{align*}
			V_{\delta'}&\defeq\{(\sigma,\delta)\mid \delta=\delta'\} \text{, and}\\
			V_{\delta',j}&\defeq\{(\sigma,\delta)\mid \delta=\delta',\sigma(n)=j\}.
		\end{align*}
		We can see that $V_{\delta'}$ is the set of regions in the orthant of $\delta'$, which is visualised in \cref{exampleOrthant}. %\veronika{Vorschlag: Ergänzung}
		The induced subgraphs $\mT(\Bn)[V_{\delta'}]$ is canonically isomorphic to the graph $\mT(\mA_{n-1})$ because every orthant in $\Bn$ is isomorphic to an arrangement of type $\An.$	The induced subgraph $\mT(\Bn)[V_{\delta',j}]$ is canonically isomorphic to the graph $\mT(\mA_{n-2}),$ as the corresponding regions can be seen as the regions in an orthant of $\mathcal{B}_{n-1}$, which arises as the restriction of $\Bn$ to the hyperplane $H_{e_j}$.
		%can be seen as the regions in an orthant of the hyperplane $H_{e_k},$ which is isomorphic to $\mathcal{B}_{n-1}$, as all $H_{e_i\pm e_j}$ and $H_{e_i}$ for $i,j\in[n]\setminus\{k\}$ have non-trivial intersection with $H_{e_k}$.
		Edges of these subgraphs are of type $H_{e_i-e_j}$ and $H_{e_k+e_l}$ for some $i,j,k,l\in [n]$, where the signs depend on the $\delta'$.
	\end{bem}
	\begin{figure}
		\centering
		% https://q.uiver.app/#q=WzAsMTIsWzIsMSwiVl97KCssLSl9Il0sWzQsMSwiVl97KCssKyl9Il0sWzEsNCwiVl97KC0sLSl9Il0sWzUsNCwiVl97KC0sKyl9Il0sWzMsMCwiSF97ZV8yfSJdLFszLDZdLFswLDNdLFs2LDMsIkhfe2VfMX0iXSxbMCwwXSxbNiw2LCJIX3tlXzErZV8yfSJdLFswLDZdLFs2LDAsIkhfe2VfMS1lXzJ9Il0sWzQsNSwiIiwwLHsic3R5bGUiOnsiaGVhZCI6eyJuYW1lIjoibm9uZSJ9fX1dLFs2LDcsIiIsMCx7InN0eWxlIjp7ImhlYWQiOnsibmFtZSI6Im5vbmUifX19XSxbOCw5LCIiLDAseyJzdHlsZSI6eyJoZWFkIjp7Im5hbWUiOiJub25lIn19fV0sWzEwLDExLCIiLDAseyJzdHlsZSI6eyJoZWFkIjp7Im5hbWUiOiJub25lIn19fV0sWzYsNCwiIiwxLHsiY3VydmUiOi01LCJjb2xvdXIiOlsyMzQsMTAwLDYwXSwic3R5bGUiOnsiaGVhZCI6eyJuYW1lIjoibm9uZSJ9fX1dLFs2LDQsIiIsMSx7Im9mZnNldCI6LTUsImN1cnZlIjo1LCJjb2xvdXIiOlsyMzQsMTAwLDYwXSwic3R5bGUiOnsiaGVhZCI6eyJuYW1lIjoibm9uZSJ9fX1dLFs0LDcsIiIsMSx7Im9mZnNldCI6LTUsImN1cnZlIjo1LCJjb2xvdXIiOlsyMzQsMTAwLDYwXSwic3R5bGUiOnsiaGVhZCI6eyJuYW1lIjoibm9uZSJ9fX1dLFs2LDUsIiIsMSx7ImN1cnZlIjo1LCJjb2xvdXIiOlsyMzQsMTAwLDYwXSwic3R5bGUiOnsiaGVhZCI6eyJuYW1lIjoibm9uZSJ9fX1dLFs1LDYsIiIsMSx7Im9mZnNldCI6LTUsImN1cnZlIjo1LCJjb2xvdXIiOlsyMzQsMTAwLDYwXSwic3R5bGUiOnsiaGVhZCI6eyJuYW1lIjoibm9uZSJ9fX1dLFs1LDcsIiIsMSx7ImN1cnZlIjo1LCJjb2xvdXIiOlsyMzQsMTAwLDYwXSwic3R5bGUiOnsiaGVhZCI6eyJuYW1lIjoibm9uZSJ9fX1dLFs0LDcsIiIsMSx7ImN1cnZlIjotNSwiY29sb3VyIjpbMjM0LDEwMCw2MF0sInN0eWxlIjp7ImhlYWQiOnsibmFtZSI6Im5vbmUifX19XSxbNSw3LCIiLDEseyJvZmZzZXQiOjUsImN1cnZlIjotNSwiY29sb3VyIjpbMjM0LDEwMCw2MF0sInN0eWxlIjp7ImhlYWQiOnsibmFtZSI6Im5vbmUifX19XV0=
		% \includegraphics[width=0.45\linewidth]{orthantsBn.png}
		\begin{tikzpicture}[scale=0.8]
			\coordinate (TL) at (-4,4);
			\coordinate (BL) at (-4,-4);
			\coordinate (TR) at (4,4);
			\coordinate (BR) at (4,-4);
			\coordinate (L) at (-4,0);
			\coordinate (R) at (4,0);
			\coordinate (T) at (0,4);
			\coordinate (B) at (0,-4);
			
			\coordinate (vari1) at (0.4, 0.7);
			\coordinate (vari2) at (-0.4, 0.7);		
			
			\coordinate (shiftright) at (0.4,0);
			\coordinate (shiftup) at (0,0.4);
			
			\draw (L) -- (R) node[right]{$H_{e_1}$};
			\draw (B) -- (T) node[above]{$H_{e_2}$};
			\draw (BL) -- (TR) node[anchor=south west]{$H_{e_1-e_2}$};
			\draw (TL) -- (BR) node[anchor=north west]{$H_{e_1+e_2}$};
			
			\draw[blue] ($(L)+(shiftup)$) to [bend right=35] ($(T)-(shiftright)$);
			\draw[blue] ($(L)+(shiftup)$) to [bend left=35] ($(T)-(shiftright)$);
			\draw[blue] ($0.5*(TL)+(vari1)$) node{$V_{(+,-)}$};
			
			\draw[blue] ($(R)+(shiftup)$) to [bend right=35] ($(T)+(shiftright)$);
			\draw[blue] ($(R)+(shiftup)$) to [bend left=35] ($(T)+(shiftright)$);
			\draw[blue] ($0.5*(TR)+(vari2)$) node{$V_{(+,+)}$};
			
			\draw[blue] ($(L)-(shiftup)$) to [bend right=35] ($(B)-(shiftright)$);
			\draw[blue] ($(L)-(shiftup)$) to [bend left=35] ($(B)-(shiftright)$);
			\draw[blue] ($0.5*(BL)-(vari2)$) node{$V_{(-,-)}$};
			
			\draw[blue] ($(R)-(shiftup)$) to [bend right=35] ($(B)+(shiftright)$);
			\draw[blue] ($(R)-(shiftup)$) to [bend left=35] ($(B)+(shiftright)$);
			\draw[blue] ($0.5*(BR)-(vari1)$) node{$V_{(-,+)}$};
		\end{tikzpicture}
		
		\caption{An example of the orthants for $\mB_2$}
		\label{exampleOrthant}
	\end{figure}
	
	\begin{defi}
		The hyperplane arrangement of type $\Dns$ with $n,s\in\mathbb{N}$ and $0\leq s\leq n$ is defined as\[\Dns = \set{H_{e_i\pm e_j} \mid 1 \leq i < j \leq n} \cup \set{H_{e_i} \mid 1 \leq i \leq s}.\]
		
		The arrangement $\Dns$ can be constructed by deleting $n-s$ coordinate hyperplanes from the arrangement of type $\Bn$ and thus, $\mD_{n,n}\cong \Bn$. They also arise as a restriction of $\mD_{n+s}$ to $s$ hyperplanes.
	\end{defi}
	\begin{theorem}[{\cite{conway}}]\label{thm:csw}
		The arrangements of type $\An$, $\Bn$, and $\Dn$ admit a Hamiltonian cycle.
	\end{theorem}
	
	It remains to show that the arrangements of type $\Dns$ also have Hamiltonian cycles. In the proof of \cref{thm:csw}, we connect Hamiltonian cycles in subgraphs to a common Hamiltonian cycle for the graph, as shown in \cref{glueingpic} and described in the following \cref{glueing}.
	\begin{lemm}\label{glueing}
		Let $G$ be a graph with subgraphs $G_1$ and $G_2$ that partition the set of vertices of $G$. Assume that $G_1$ and $G_2$ both have a Hamiltonian cycle and there exist edges $e_1\in E(G_1)$, $e_2\in E(G_2)$ and $f_1,f_2 \in E(G)\setminus(E(G_1)\cup E(G_2))$, such that $e_1$ and $e_2$ are used in the Hamiltonian cycles of $G_1$ and $G_2$, respectively. Additionally, assume that the edges $e_1,f_1,e_2,f_2$ form a closed cycle in $G$. Then, we can combine the two Hamiltonian cycles of the subgraphs into one Hamiltonian cycle of $G$.
	\end{lemm}
	\begin{proof}
		Let $G$ be a graph with subgraphs $G_1$ and $G_2$ fulfilling the requirements of the lemma. Let $C_1$ and $C_2$ be Hamiltonian cycles of $G_1$ and $G_2$ respectively. We now can remove the edges $e_1$ and $e_2$ from $C_1$ and $C_2$, and then connect $C_1$ and $C_2$ into one large cycle $C$ by adding the edges $f_1$ and $f_2$, as it can be seen in \cref{glueingpic}. Since $G_1$ and $G_2$ partition the vertex set of $G$, the cycle $C$ is a Hamiltonian cycle for $G$.
	\end{proof}
	We call the 4-cycles used in \cref{glueing} \defn{quadrilaterals}.
	\begin{figure}[htp]
		\centering
		\begin{tikzpicture}[x=0.75pt,y=0.75pt,yscale=1,xscale=1,scale=5]
			\draw (0,0)--(10,0);
			\draw [dashed](0,0)--(0,10) [color={rgb, 255:red, 155; green, 200; blue, 0 }  ,opacity=1 ];
			\draw [dashed](10,0)--(10,10) [color={rgb, 255:red, 200; green, 155; blue, 0 }  ,opacity=1 ];
			\draw (0,10)--(10,10);
			
			\draw(0,0).. controls (-10,-10) and (-10,20) ..(0,10)  [color={rgb, 255:red, 155; green, 200; blue, 0 }  ,opacity=1 ] ;
			\draw(10,0).. controls (20,-10) and (20,20) ..(10,10)  [color={rgb, 255:red, 200; green, 155; blue, 0 }  ,opacity=1 ] ;

			\fill (0,0)  circle (0.5);
			\fill (10,0) circle (0.5);
			\fill (0,10) circle (0.5);
			\fill (10,10) circle (0.5);
			
			\draw (1.5,1);% node {$v_1$};
			\draw (8,1) ;%node {$v_2$};
			\draw (1.5,8); %node {$v_4$};
			\draw (8,8) ;%node {$v_3$};
			
			\draw (5,-2) node {$f_2$};
			\draw (11.5,5) node {$e_2$};
			\draw (5,11.5) node {$f_1$};
			\draw (-2,5) node {$e_1$};
			
		\end{tikzpicture}
		\caption{Connecting two smaller cycles in green and brown via the black edges $f_1,f_2$.}
		\label{glueingpic}
	\end{figure}
	
	\begin{theorem}\label{DnsRes}
		All arrangements of type $\Dns$ for $n\geq2$ admit a Hamiltonian cycle.
	\end{theorem}
	\begin{proof}
		Because $\mD_{n,n}\cong \Bn$ and $\mD_{n,0}\cong \Dn$, we only need to consider the arrangements $\Dns$ for all $1\le s\le (n-1)$. At first, we consider some base cases of $\Dns$. Trivially, $\mD_{2,1}$ admits a Hamiltonian cycle. The arrangements $\mathcal D_{3,s}$, $\mathcal{D}_{4,s}$ and $\mathcal{D}_{5,s}$ admit a Hamiltonian cycle and their Hamiltonian cycles can be found in \cref{appendix:sporadic}.
		
		Now, let $n\geq 6$. We prove that $\Dns$ is Hamiltonian for all $1\le s\le (n-1)$. To begin, we make some observations about $\Bn$ and $\Dns$. As described above in \cref{bemBn}, %Aktuell ist das im Fließtext, also kann man hier keine \cref einfügen
		%\tobias{TODO: Hier irgendeine einfache referenzierbare Form (z.B. \textbackslash eqref oder Kommentar) einfügen}
		every vertex in $\mT(\Bn)$ has exactly one adjacent edge of type $H_{e_j}$ for some $j\in[n]$. We can partition the vertices of $\mT(\Bn)$ into the sets $V_{\delta,j}$ for $\delta \in \Z_2^n$ and $j\in[n]$.
		The tope graph of the arrangement $\Dns$ arises as described in \cref{contracHam}, by contracting edges of type $H_{e_k}$ for all $k\in\{s+1,\ldots, n\}$. In this special case, it means that the subgraphs $G[V_{\delta,k}]$ and $G[V_{\delta',k}]$, where $\delta'$ only differs from $\delta$ in $k$ and which are isomorphic to $\mT(\mA_{n-2})$ get contracted into one subgraph isomorphic to $\mT(\mA_{n-2})$.
		
		Under the map $\pi\colon\mR(\Bn)\rightarrow\mR(\Dns)$ from \cref{defi-surjSubArr}, each region in $\Dns$ is the image of at most two regions of $\Bn$. If an edge $e$ between two regions from $\Bn$ is of type $H_{e_k}$ for $k\in\{s+1,\ldots, n\}$, then the two regions corresponding to the adjacent vertices of $e$ have the same image under $\pi$. On all other regions, $\pi$ is bijective.
		
		With these observations, also the tope graph $G$ of $\mD_{n,s}$ can be partitioned into subgraphs isomorphic to $\mT(\mA_{n-2})$. Each of these graphs is Hamiltonian by \cref{thm:csw}, and we pick the same Hamiltonian cycle on each subgraph. In this way, the Hamiltonian cycles of two adjacent subgraphs are mirrored along their separating hyperplane. Thus, the identification of them yields a unique well-defined Hamiltonian cycle.
		
		Now, we describe how to glue the Hamiltonian cycles of these subgraphs together. Therefor, we first construct a graph $\mI(G)$ by contracting all vertices  of each $G[V_{\delta,k}]$ subgraph into one vertex, respectively. The resulting graph has a vertex for each $G[V_{\delta,k}]$ subgraph and edges between them, if the $G[V_{\delta,k}]$ subgraphs are adjacent in $\mT(\Dns)$. This graph for $\mD_{2,1}$ can be seen in \cref{exampleIncident}. We show that for each spanning tree in $\mI(G)$, we can construct a Hamiltonian cycle for $\mT(\Dns)$ by connecting the Hamiltonian cycles of the $G[V_{\delta,k}]$ subgraphs along the edges of the spanning tree.
		%	\begin{figure}[tp]
			%		\centering\small
			%		% https://q.uiver.app/#q=WzAsMTAsWzMsMCwiSF97ZV8yfSJdLFszLDZdLFswLDNdLFs2LDMsIkhfe2VfMX0iXSxbMywxLCJcXG1hdGhjYWx7QX1fe24tMn0iXSxbMSwyLCJcXG1hdGhjYWx7QX1fe24tMn0iXSxbMSw0LCJcXG1hdGhjYWx7QX1fe24tMn0iXSxbMyw1LCJcXG1hdGhjYWx7QX1fe24tMn0iXSxbNSw0LCJcXG1hdGhjYWx7QX1fe24tMn0iXSxbNSwyLCJcXG1hdGhjYWx7QX1fe24tMn0iXSxbMCwxLCIiLDAseyJzdHlsZSI6eyJib2R5Ijp7Im5hbWUiOiJkb3R0ZWQifSwiaGVhZCI6eyJuYW1lIjoibm9uZSJ9fX1dLFsyLDMsIiIsMCx7InN0eWxlIjp7ImhlYWQiOnsibmFtZSI6Im5vbmUifX19XSxbNSw0LCIiLDAseyJzdHlsZSI6eyJoZWFkIjp7Im5hbWUiOiJub25lIn19fV0sWzUsNiwiIiwyLHsic3R5bGUiOnsiaGVhZCI6eyJuYW1lIjoibm9uZSJ9fX1dLFs2LDcsIiIsMix7InN0eWxlIjp7ImhlYWQiOnsibmFtZSI6Im5vbmUifX19XSxbNyw4LCIiLDIseyJzdHlsZSI6eyJoZWFkIjp7Im5hbWUiOiJub25lIn19fV0sWzgsOSwiIiwyLHsic3R5bGUiOnsiaGVhZCI6eyJuYW1lIjoibm9uZSJ9fX1dLFs0LDksIiIsMCx7InN0eWxlIjp7ImhlYWQiOnsibmFtZSI6Im5vbmUifX19XV0=
			%		\[\begin{tikzcd}
				%			&&& {H_{e_2}} \\
				%			&&& {\mathcal{A}_{n-2}} \\
				%			& {\mathcal{A}_{n-2}} &&&& {\mathcal{A}_{n-2}} \\
				%			{} &&&&&& {H_{e_1}} \\
				%			& {\mathcal{A}_{n-2}} &&&& {\mathcal{A}_{n-2}} \\
				%			&&& {\mathcal{A}_{n-2}} \\
				%			&&& {}
				%			\arrow[dotted, no head, from=1-4, to=7-4]
				%			\arrow[no head, from=2-4, to=3-6]
				%			\arrow[no head, from=3-2, to=2-4]
				%			\arrow[no head, from=3-2, to=5-2]
				%			\arrow[no head, from=4-1, to=4-7]
				%			\arrow[no head, from=5-2, to=6-4]
				%			\arrow[no head, from=5-6, to=3-6]
				%			\arrow[no head, from=6-4, to=5-6]
				%		\end{tikzcd}\]
			%		\caption{The graph $\mI(\mT(\mD_{2,1}))$ that shows how the $\mA_{n-2}$ subgraphs are incident to each other}
			%		\label{exampleIncident}
			%	\end{figure}
		Therefore, we have to show that such two adjacent subgraphs that are isomorphic to $\mT(\mA_{n-2})$ have a quadrilateral between them, and that we can connect them in such a way that these quadrilaterals do not overlap, that is, we obtain a Hamiltonian cycle for $\Dns$.
		
		Firstly, consider two neighbouring subgraphs $G[V_{\delta,j}]$ and $G[V_{\delta',j}]$ for some $j\in[s]$. These subgraphs are isomorphic to the tope graph of $\mA_{n-2}$ and have the same Hamiltonian cycle.
		Choose a permutation $\sigma_1\in \mathfrak{S}_n$ with $\sigma_1(n)=j$. Both graphs have a vertex that corresponds to $\sigma_1$. Then the vertices $(\sigma_1,\delta)$ and $(\sigma_1,\delta')$ have neighbouring vertices $(\sigma_2,\delta)$ and $(\sigma_2,\delta')$ in the Hamiltonian cycle, for a permutation $\sigma_2\in \mathfrak{S}_n$. These four vertices form a quadrilateral via the edge of type $H_{e_j}$. Therefore, we can connect the Hamiltonian cycles of any two $\mA_{n-2}$ subgraphs incident to edges of type $H_{e_j}$  by \cref{glueing}.
		
		Secondly, we consider two subgraphs $G[V_{\delta,j}]$ and $G[V_{\delta,k}]$ for a fixed $\delta\in\Z_2^n$ and $j,k\in[s]$ with $j\neq k$. Each vertex $(\sigma,\delta)\in V_{\delta,j}$ has a neighbour in $V_{\delta,k}$ if and only if $\sigma(n-1)=k$, as described before. In particular, the set of vertices in $V_{\delta,j}$ that have neighbours in $V_{\delta,k}$ is disjoint from the set of vertices that have neighbours in $V_{\delta,l}$ for $k\neq l$. Therefore, $V_{\delta,j}$ has disjoint sets of quadrilaterals for $V_{\delta,k}$ and $V_{\delta,l}$ respectively.
		
		Choose a vertex $(\sigma,\delta)\in V_{\delta,j}$ with $\sigma(n-1)=k$. In the Hamiltonian cycle of $G[V_{\delta,j}]$, the vertex $(\sigma,\delta)$ has at least one neighbour $(\sigma',\delta)$, such that the neighbours of $(\sigma,\delta)$ and $(\sigma',\delta)$ are neighbours in the Hamiltonian cycle in $G[V_{\delta,k}]$. This holds, because at most one incident edge of $(\sigma,\delta)$ can change the position of $\sigma(n-1)=k$, and we chose the same Hamiltonian cycle for $G[V_{\delta,j}]$ and $G[V_{\delta,k}]$. The chosen four vertices yield a quadrilateral, and we can use \cref{glueing}.
		
		Since we chose $(\sigma,\delta)\in V_{\delta,j}$ to be arbitrary, and for each $(\sigma,\delta)$ there is at least one suitable neighbour, we have at least $\tfrac{(n-2)!}{4}$ disjoint quadrilaterals along $G[V_{\delta,j}]$ and $G[V_{\delta,k}]$. This is because every vertex is contained in at least one quadrilateral with an adjacent vertex, and two quadrilaterals can share a common edge.
		
		Thirdly, consider a subgraph $G[V_{\delta,j}]$ with $j>s$. In $\mT(\Dns)$, this subgraph is the graph that resulted by contracting two subgraphs $G[V_{\delta,j}]$ and $G[V_{\delta',j}]$. We have to make sure that we use distinct quadrilaterals while connecting $G[V_{\delta,j}]$ to its neighbours. Otherwise, in the resulting graph after connecting Hamiltonian cycles and contracting the two subgraphs, there is a vertex of degree 3 as sketched in \cref{same4eck} and then the resulting union of edges cannot be a cycle anymore.
		
		\begin{figure}[tp]
			\centering\small
			% https://q.uiver.app/#q=WzAsMTAsWzMsMCwiSF97ZV8yfSJdLFszLDZdLFswLDNdLFs2LDMsIkhfe2VfMX0iXSxbMywxLCJcXG1hdGhjYWx7QX1fe24tMn0iXSxbMSwyLCJcXG1hdGhjYWx7QX1fe24tMn0iXSxbMSw0LCJcXG1hdGhjYWx7QX1fe24tMn0iXSxbMyw1LCJcXG1hdGhjYWx7QX1fe24tMn0iXSxbNSw0LCJcXG1hdGhjYWx7QX1fe24tMn0iXSxbNSwyLCJcXG1hdGhjYWx7QX1fe24tMn0iXSxbMCwxLCIiLDAseyJzdHlsZSI6eyJib2R5Ijp7Im5hbWUiOiJkb3R0ZWQifSwiaGVhZCI6eyJuYW1lIjoibm9uZSJ9fX1dLFsyLDMsIiIsMCx7InN0eWxlIjp7ImhlYWQiOnsibmFtZSI6Im5vbmUifX19XSxbNSw0LCIiLDAseyJzdHlsZSI6eyJoZWFkIjp7Im5hbWUiOiJub25lIn19fV0sWzUsNiwiIiwyLHsic3R5bGUiOnsiaGVhZCI6eyJuYW1lIjoibm9uZSJ9fX1dLFs2LDcsIiIsMix7InN0eWxlIjp7ImhlYWQiOnsibmFtZSI6Im5vbmUifX19XSxbNyw4LCIiLDIseyJzdHlsZSI6eyJoZWFkIjp7Im5hbWUiOiJub25lIn19fV0sWzgsOSwiIiwyLHsic3R5bGUiOnsiaGVhZCI6eyJuYW1lIjoibm9uZSJ9fX1dLFs0LDksIiIsMCx7InN0eWxlIjp7ImhlYWQiOnsibmFtZSI6Im5vbmUifX19XV0=
			\[\begin{tikzcd}
				&&& {H_{e_2}} \\
				&&& {\mathcal{A}_{n-2}} \\
				& {\mathcal{A}_{n-2}} &&&& {\mathcal{A}_{n-2}} \\
				{} &&&&&& {H_{e_1}} \\
				& {\mathcal{A}_{n-2}} &&&& {\mathcal{A}_{n-2}} \\
				&&& {\mathcal{A}_{n-2}} \\
				&&& {}
				\arrow[dotted, no head, from=1-4, to=7-4]
				\arrow[no head, from=2-4, to=3-6]
				\arrow[no head, from=3-2, to=2-4]
				\arrow[no head, from=3-2, to=5-2]
				\arrow[no head, from=4-1, to=4-7]
				\arrow[no head, from=5-2, to=6-4]
				\arrow[no head, from=5-6, to=3-6]
				\arrow[no head, from=6-4, to=5-6]
			\end{tikzcd}\]
			\caption{The graph $\mI(\mT(\mD_{2,1}))$ that shows how the $\mA_{n-2}$ subgraphs are incident to each other}
			\label{exampleIncident}
		\end{figure}
		
		\begin{figure}[b]
			\centering
			\begin{tikzpicture}[x=0.75pt,y=0.75pt,yscale=1,xscale=1,scale=10]
				
				\draw [color=green] (4,1)--(4,4)--(4,7);				
				\draw [color=green] (6,1)--(6,4)--(6,7);				
				\draw [color=green] (4,1)--(4,4)--(4,7)--(4,10);								
				\draw [color=green] (6,1)--(6,4)--(6,7)--(6,10);				
				\draw [color=blue] (4,4)--(4,7);				
				\draw [color=blue] (6,4)--(6,7);
				\draw [color=blue] (3.1,4.3)--(3.1,7.3);
				\draw [color=blue] (6.9,4.3)--(6.9,7.3);			
				\draw [color=blue](5,4.7)--(5,7.7);				
				
				\fill (4,4) circle (0.2pt);				
				\fill (6,4) circle (0.2pt);				
				\fill (4,1) circle (0.2pt);				
				\fill (6,1) circle (0.2pt);				
				\fill (4,7) circle (0.2pt);				
				\fill (6,7) circle (0.2pt);			
				\fill (4,10) circle (0.2pt);				
				\fill (6,10) circle (0.2pt);

				\draw[dashed] (4,4) -- (6,4);				
				\draw (4,4) .. controls (0,5) and (10,5) .. (6,4);			
				\draw [dashed] (4,1) -- (6,1);				
				\draw (4,1) .. controls (0,2) and (10,2) .. (6,1);				
				\draw [dashed](4,7) -- (6,7);			
				\draw (4,7) .. controls (0,8) and (10,8) .. (6,7);				
				\draw [dashed] (4,10) -- (6,10);				
				\draw (4,10) .. controls (0,11) and (10,11).. (6,10);
				
				\draw  (8,6)--(12,6);				
				\draw (8,5.8)--(12,5.8);			
				\draw (12.1,5.9)--(11.8,5.6);				
				\draw (12.1,5.9)--(11.8,6.2);	
				
				\draw [color=green] (14,3)--(14,6)--(14,9);			
				\draw [color=green] (16,3)--(16,6)--(16,9);			
				\fill (14,3) circle (0.2pt);				
				\fill (16,3) circle (0.2pt);			
				\fill (14,6) circle (0.2pt);		
				\fill (16,6) circle (0.2pt);			
				\fill (14,9) circle (0.2pt);			
				\fill (16,9) circle (0.2pt);

				\draw(14,3) .. controls (10,4) and (20,4) .. (16,3);					
				\draw  (14,6) .. controls (10,7) and (20,7) .. (16,6);				
				\draw  (14,9) .. controls (10,10) and (20,10) .. (16,9);

			\end{tikzpicture}
			\caption{On the left, the four black Hamiltonian cycles can be connected into two Hamiltonian cycles via the green quadrilaterals and two graphs get contracted into one.\\On the right, we see that the union of edges cannot be a Hamiltonian cycle anymore.}
			\label{same4eck}
		\end{figure}
		For $n\ge6$, we have at least $\tfrac{(n-2)!}{4}\ge\tfrac{4!}{4}=6$ disjoint quadrilaterals between each $G[V_{\delta,j}]$ and $G[V_{\delta,k}]$, and each subgraph $G[V_{\delta,j}]$ gets identified with at most one other subgraph $G[V_{\delta',j}]$. So, we have enough disjoint quadrilaterals to choose disjoint ones for each pair $G[V_{\delta,j}]$, $G[V_{\delta,k}]$ and $G[V_{\delta',j}]$, $G[V_{\delta',k}]$.
		
		In summary, for each of the subgraphs isomorphic to $\mT(\mA_{n-2}),$ we can connect its Hamiltonian cycle to each Hamiltonian cycle of such an incident subgraph. We choose an arbitrary spanning tree on $\mI(G)$ and connect the Hamiltonian cycles of the subgraphs isomorphic to $\mT(\mA_{n-2})$ along the spanning tree to create a common Hamiltonian cycle for $\mT(\Dns)$.% \tobias{Todo: Grundidee gut, Struktur eventuell suboptimal. Veronika schreibt eine Alternativfassung und dann können wir vergleichen. Die Spanning Trees bleiben in jedem Fall irgendwie drin.}
		%Raman meinte der Aufbau mit den Spanning Trees hat den Beweis deutlich verständlicher für ihn gemacht, deswegen wäre ich dafür den drin zu behalten. TODO: Darüber sollten wir auch noch reden - die Struktur kann man da trotzdem verbessern
	\end{proof}
	
	\begin{proof}[Proof of \cref{allResRefl}]
		This follows directly from \cref{sporadicRest} and \cref{DnsRes}.
	\end{proof}
	
	\section{3-dimensional simplicial arrangements}\label{section4}
	
	Trivially, all 2-dimensional hyperplane arrangements are simplicial and admit a Hamiltonian cycle. However, in dimension 3 and above, we still not completely understand simplicial hyperplane arrangements. For the 3-dimensional simplicial hyperplane arrangements, there exists the Grünbaum--Cuntz catalogue \cite{CEL}. It incorporates 95 sporadic arrangements and three infinite families of simplicial arrangements \cite{grunbaumcatalog,cuntz2022greedy,27lines}. It is conjectured that the Grünbaum--Cuntz catalogue is complete up to combinatorial isomorphism \cite{cuntz2022greedy,grunbaumcatalog}. 
	
	So far, all simplicial hyperplane arrangements that we studied admit a Hamiltonian cycle. Also, the tope graph of simplicial hyperplane arrangements is a bipartite graph with odd-even invariance zero. This is a natural necessary condition for having a Hamiltonian cycle \cite{kemper2018odd}. Therefore, the question of Hamiltonian cycles in simplicial hyperplane arrangements seems natural.
	
	\begin{theorem}\label{3dsimp}
		All simplicial hyperplane arrangements in the Grünbaum--Cuntz catalogue admit a Hamiltonian cycle.
	\end{theorem}
	
	We constructed the tope graphs of all 95 sporadic arrangements and found Hamiltonian cycles for all of them, leading to the following theorem. All the constructed graphs with their cycles can be found in Appendix \cref{appendix:sporadic}.
	
	\begin{theorem}\label{95sporHam}
		All 95 sporadic arrangements of the Grünbaum--Cuntz catalogue admit a Hamiltonian cycle. 
	\end{theorem}
	Grünbaum identified 3 infinite families, called $\mR(0)$, $\mR(1)$, and $\mR(2)$ of arrangements in $\R^3$ \cite{grunbaumcatalog}. We represent $\mR(0)$, $\mR(1)$, and $\mR(2)$ in the real projective plane. This consists of all the points of the real Euclidean plane, together with the \emph{line at infinity}, that is, the set of \emph{points at infinity}, which can be interpreted as the set of lines passing through the origin of the Euclidean plane. This way of representing those projective arrangements avoids the need of an algebraic representation and makes the visual verification of the simplicial character of the arrangements that we are considering easy.
	\begin{defi}[The family $\mR(0)$]\label{def:r0}
		The hyperplane arrangements of the family $\mR(0)$ are called the \defn{near-pencil} arrangements.
		These arrangements consist of $m$ hyperplanes for $m \geq 3$, where $m-1$ hyperplanes all intersect in one line, called the \emph{pencil}, and the last hyperplane does not contain that line.
	\end{defi}
	\begin{figure}[ht]
		\centering
		\begin{tikzpicture}[dot/.style={circle,inner sep=1pt,fill,label={#1},name=#1},
			extended line/.style={shorten >=-#1,shorten <=-#1},
			extended line/.default=4cm,
			scale = .2]
			
			\clip (0,0) circle (10);
			\draw[extended line] (10,0)--(-10,0);
			\draw[extended line] (0,10)--(0,-10);
			\draw[extended line] (-8,6)--(8,-6);
			\draw[extended line] (8,6)--(-8,-6);
			\draw[extended line] (6,8)--(-6,-8);
			
			%\draw [dashed] (-5,10)--(-3,-10);
			\draw (5,10)--(5,-10);
			
		\end{tikzpicture}
		\caption{A view of the near-pencil arrangement $\mR(0)$ with 6 hyperplanes and 20 regions}
		\label{r0}
	\end{figure}
	\begin{defi}[The family $\mR(1)$]\label{def:r1}
		The hyperplane arrangements of the family $\mR(1)$ consist of $2m$ hyperplanes for $m \geq 3$. Starting with a regular $m$-gon in the real projective plane that does not intersect the line at infinity, we obtain $\mR(1)$ by taking the $m$ hyperplanes determined by the edges of the $m$-gon and the $m$ hyperplanes determined by the axes of mirror symmetry of the $m$-gon.
	\end{defi}
	\begin{figure}[htp]
		\centering
		\begin{tikzpicture}[dot/.style={circle,inner sep=1pt,fill,label={#1},name=#1},
			extended line/.style={shorten >=-#1,shorten <=-#1},
			extended line/.default=4cm,scale=.6]
			\coordinate (A) at (0,0);
			\coordinate (B) at (1,0);
			\coordinate (C) at (1.31,0.95);
			\coordinate (D) at (.5,1.54);
			\coordinate (E) at (-.31,0.95);
			\begin{scope}
				\clip (.5,.688) circle (3.3);
				\draw[extended line] (A) -- ($(C)!.5!(D)$);
				\draw[extended line] (B) -- ($(D)!.5!(E)$);
				\draw[extended line] (C) -- ($(E)!.5!(A)$);
				\draw[extended line] (D) -- ($(A)!.5!(B)$);
				\draw[extended line] (E) -- ($(B)!.5!(C)$);
				\draw[extended line] (A) -- (B);
				\draw[extended line] (B) -- (C);
				\draw[extended line] (C) -- (D);
				\draw[extended line] (D) -- (E);
				\draw[extended line] (E) -- (A);
			\end{scope}
		\end{tikzpicture}
		\caption{A view of an arrangement from the $\mR(1)$ family with 10 hyperplanes and 60 regions}
		\label{r1}
	\end{figure}
	\begin{defi}[The family $\mR(2)$]\label{def:r2}
		The hyperplane arrangements of the family $\mR(2)$ consist of $4m+1$ hyperplanes for $m \geq 2$. A hyperplane arrangement in $\mR(2)$ with $4m+1$ hyperplanes is obtained from the one in $\mR(1)$ that has $4m$ hyperplanes by adding the \emph{line at infinity} in the model of the projective plane.
	\end{defi}	
	\begin{figure}[thp]
		\centering
		\begin{tikzpicture}[dot/.style={circle,inner sep=1pt,fill,label={#1},name=#1},
			extended line/.style={shorten >=-#1,shorten <=-#1},
			extended line/.default=4cm,scale=.6]
			\coordinate (A) at (0,0);
			\coordinate (B) at (1,0);
			\coordinate (C) at (1.5,0.87);
			\coordinate (D) at (1,1.73);
			\coordinate (E) at (0,1.73);
			\coordinate (F) at (-.5,.87);
			\begin{scope}
				\clip (.5,.866) circle (3);
				\draw[extended line] (A) -- (B);
				\draw[extended line] (B) -- (C);
				\draw[extended line] (C) -- (D);
				\draw[extended line] (D) -- (E);
				\draw[extended line] (E) -- (F);
				\draw[extended line] (F) -- (A);
				
				\draw[extended line] (A) -- (D);
				\draw[extended line] (B) -- (E);
				\draw[extended line] (C) -- (F);
				
				\draw[extended line] ($(A)!.5!(B)$) -- ($(D)!.5!(E)$);
				\draw[extended line] ($(B)!.5!(C)$) -- ($(E)!.5!(F)$);
				\draw[extended line] ($(C)!.5!(D)$) -- ($(A)!.5!(F)$);
			\end{scope}
			\draw (3.6,.866) arc (0:90:3.1) node[midway, anchor = south west]{$\infty$};
		\end{tikzpicture}
		\caption{A view of an arrangement from the $\mR(0)$ family with 25 hyperplanes and 96 regions}
		\label{r2}
	\end{figure}
	\begin{bsp}
		\cref{r0,r1,r2} show representations of arrangements of each infinite family. Each is obtained by intersecting the arrangement with a generic affine hyperplane $H$ and showing the resulting lines. A parallel linear plane to $H$ is represented as the line at infinity.
		In \cref{r0}, the diagram shows 3 bounded and 14 unbounded regions. The unbounded regions extend beyond $H$, and the bounded regions each have a corresponding bounded region on the other side of $H$. Therefore, we obtain 20 regions in total.
		In \cref{r1}, the same reasoning yields 20 bounded and 20 unbounded regions, for a total of 60 regions.
		In \cref{r2}, we have a parallel hyperplane to $H$. Therefore, each region we see in \cref{r2} has a corresponding region beyond this hyperplane, and thus, we have 96 regions in total (48 regions on each side).
	\end{bsp}
	\begin{lemm}\label{superr0r1r2}
		The hyperplane arrangements in the three infinite families $\mR(0), \mR(1)$, and $\mR(2)$ are supersolvable.
	\end{lemm}
	\begin{proof}
		It is proven in \cite[Lemma 4.7]{cuntzsupersolvablesimplicial} that $\mR(1)$ and $\mR(2)$ are supersolvable arrangements.
		
		Consider a near-pencil arrangement. Since we are in dimension 3, it suffices to find a modular coatom in the intersection lattice. It is easy to check that the intersection line where $m-1$ of the hyperplanes all intersect is modular. Therefore, its intersection lattice is supersolvable. 
	\end{proof}
	\begin{theorem}\label{3infiHam}
		The arrangements $\mR(0), \mR(1)$, and $\mR(2)$ of the Grünbaum--Cuntz catalogue admit a Hamiltonian cycle. 
	\end{theorem}
	\begin{proof}
		By \cref{superr0r1r2} the hyperplane arrangements in the three infinite families are supersolvable. From \cref{thm:supersolvableishamiltonian}, it follows that they admit a Hamiltonian cycle
	\end{proof}
	
	\begin{proof}[Proof of \cref{3dsimp}]
		This follows directly from \ref{95sporHam} and \ref{3infiHam}.
	\end{proof}
	
	\clearpage
	\appendix

\section{The algorithm}\label[secinapp]{appendix:algo}
We describe an algorithm to calculate the tope graph of simplicial arrangements. We use this code to check sporadic examples on Hamiltonicity and create a database of graphs for examples of simplicial arrangements, which can be found at \cite{githubdb} and uses the \texttt{SageMath} mathematics software system.

Our algorithm needs one known region of the arrangement. We can compute it from a given positive system. At first, we calculate the cone generated by all roots of the positive system. We can optimise in a generic direction to find one generating ray of the cone. By normalizing all other normal vectors to the generic direction, we get a simplicial cone in smaller dimension. Now we can inductively find all generating rays of this cone. This gives us a region of the simplicial arrangement, which is bounded by the hyperplanes of the found roots.

From this starting region, we go through the whole hyperplane arrangement, creating vertices for each newly found region and connecting these, if they are adjacent to the same wall.

Our algorithm works as follows.
We save each visited region as a tope of the starting positive system. Each time we discover a new region, we append the region with all walls (except the wall from which we are coming) to a stack. By working through the whole stack, we ensure that we have found all possible regions and walked through all possible edges of the tope graph.
In each step, we pop the last region from the stack. If there is more than one wall of this region left to check, we append it again to the stack with one wall less to check. For this wall, we create the neighbouring region.
To create a neighbouring region of our current region, we take each pair of the separating walls and one other wall of the current region. Now, we search all positive roots that are generated by the roots of the chosen walls. These roots build a 2-dimensional arrangement. We search the root, which is closest to the root of the separating wall. This closest root forms a new generating system with the negative root of the separating wall. Consequently, the walls of the neighbouring region are described by the newly found root with the negative separating wall root. If we do not find any root generated by the pair of walls, these walls stay walls of the neighbouring region.

Since we do this step for all walls of the current region and all regions are simplicial, we find all the walls of the neighbouring region.
If the new region was not already found, we add it to the tope graph with the connecting edge from the separating wall and put the region with its other walls to check on the stack, otherwise, we just add the connecting edge. This lets us create the whole tope graph of the simplicial hyperplane arrangement.

Now, we explain how we do the calculations in the iterative step. Given a region, we have a list of unit vectors and non-negative vectors which describe the positive system of the region and how the roots of the walls generate the other positive roots. We can create that list by doing a basis change operation on the positive roots given the generating system of them.

With this description, one can quickly calculate which roots are generated by two unit roots. Then we can easily calculate the closest root by taking the root with the largest quotient of their two entries. The entry of the root of a separating wall acts as the numerator for that quotient. If we look at the 2-dimensional hyperplane arrangement, we see that this characterises the closest root. At last, after finding each root for the walls of the new region, we can calculate the new list of unit vectors and non-negative vectors for the positive system of the new region from the old list of roots. Let $e_i$ and $e_j$ be the unit vectors of a pair of walls in the old region, where $e_i$ is the unit vector of the separating root. We multiply the $i$-th entry of all roots by minus one, except for the separating root itself. Let $q$ be the quotient of the new root, which we calculated before, and $d$ the denominator of this quotient. We add $q$ times its $j$-th entry to the $i$-th entry of each root and then divide its $j$-th entry by $d$. We do this for all pairs of walls, where we found new walls in the new region. This corresponds to the change of basis from the old generating root system to the new one. In conclusion, we are able to calculate every step in the algorithm described before.
\begin{theorem}
	Our algorithm computes the tope graph of a given simplicial arrangement, given the positive system of the arrangement.
\end{theorem}
The correctness of our algorithm follows from the description. The code for the algorithm can be found in the attached \texttt{SageMath} script and in \cite{githubdb}.

\section{Hamilton cycles in sporadic examples}\label[secinapp]{appendix:sporadic}

The graphs and Hamilton cycles of all arrangements of interest can be found in the attached database text file and in \cite{githubdb}.

@article{grunbaumcatalog,
	author = {Gr{\"u}nbaum, Branko},
	month = {01},
	title = {A catalogue of simplicial arrangements in the real projective plane},
	journal = {Ars Math. Contemp},
	FJournal = {Ars Mathematica Contemporanea},
	ISSN = {1855-3966},
	Volume = {2},
	Number = {1},
	Pages = {1--25},
	Year = {2009}
}

@article{conway,
	author = {Conway, J. H. and N.J.A. Sloane and A.R. Wilks},
	title = {Gray Codes for Reflection Groups},
	journal = {Graphs and Combinatorics},
	year = {1989},
	volume = {5},
	pages = {315--325}
}

@misc{sanyal22,
	title={Inscribable Fans II: Inscribed zonotopes, simplicial arrangements, and reflection groups}, 
	author={Sebastian Manecke and Raman Sanyal},
	year={2022},
	eprint={2203.11062},
	archivePrefix={arXiv},
}

@article{mu2015supersolvability,
	title={Supersolvability and freeness for $\psi$-graphical arrangements},
	author={Mu, L. and Stanley, R. P.},
	journal={Discrete \& computational geometry},
	volume={53},
	number={4},
	pages={965--970},
	year={2015},
	publisher={Springer}
}

@article{cuntz2019supersolvable,
	title={Supersolvable simplicial arrangements},
	author={Cuntz, M. and M{\"u}cksch, P.},
	journal={Advances in Applied Mathematics},
	volume={107},
	pages={32--73},
	year={2019},
	publisher={Elsevier}
}

@article{hoge2014supersolvable,
	title={Supersolvable reflection arrangements},
	author={Hoge, T. and R{\"o}hrle, G.},
	journal={Proceedings of the American Mathematical Society},
	volume={142},
	number={11},
	pages={3787--3799},
	year={2014}
}

@article{BjEdZi,
	author = {Björner, Anders and Paul H. Edelman and Günter M. Ziegler},
	title = {Hyperplane arrangements with a lattice of regions},
	journal = {Discrete \& Computational Geometry},
	year = {1990},
	pages = {263-288},
	volume = {5},
	doi = {10.1007/BF02187790}
}

@Article{27lines,
	Author = {Cuntz, Michael},
	Title = {Simplicial arrangements with up to 27 lines},
	FJournal = {Discrete \& Computational Geometry},
	Journal = {Discrete Comput. Geom.},
	ISSN = {0179-5376},
	Volume = {48},
	Number = {3},
	Pages = {682--701},
	Year = {2012},
	Language = {English},
	DOI = {10.1007/s00454-012-9423-7},
	zbMATH = {6093487},
	Zbl = {1254.52009}
}

@Article{CEL,
	Author = {Cuntz, Michael and Elia, Sophia and Labb{\'e}, Jean-Philippe},
	Title = {Congruence normality of simplicial hyperplane arrangements via oriented matroids},
	FJournal = {Annals of Combinatorics},
	Journal = {Ann. Comb.},
	ISSN = {0218-0006},
	Volume = {26},
	Number = {1},
	Pages = {1--85},
	Year = {2022},
	Language = {English},
	DOI = {10.1007/s00026-021-00555-2},
	zbMATH = {7511942},
	Zbl = {1487.52030}
}

@article{cuntz2022greedy,
	title={A greedy algorithm to compute arrangements of lines in the projective plane},
	author={Cuntz, Michael},
	journal={Discrete \& Computational Geometry},
	volume={68},
	number={1},
	pages={107--124},
	year={2022},
	publisher={Springer}
}

@misc{githubdb,
	author = {Veronika Körber and Tobias Schnieders and Jan Stricker and Jasmin Walizadeh},
	title = {Hamilton Cycles in Simplicial Arrangements, \emph{GitHub,} \url{https://github.com/Tobias271828/Hamilton-Cycles-in-Simplicial-Arrangements}},
	year = {2025},
	month = {03}
}

@book{nijenhuis2014combinatorial,
	title={Combinatorial algorithms: for computers and calculators},
	author={Nijenhuis, Albert and Wilf, Herbert S},
	year={2014},
	publisher={Elsevier}
}

@article{johnson1963generation,
	title={Generation of permutations by adjacent transposition},
	author={Johnson, Selmer M},
	journal={Mathematics of computation},
	volume={17},
	number={83},
	pages={282--285},
	year={1963},
	publisher={JSTOR}
}

@article{trotter1962algorithm,
	title={Algorithm 115: perm},
	author={Trotter, Hale F},
	journal={Communications of the ACM},
	volume={5},
	number={8},
	pages={434--435},
	year={1962},
	publisher={ACM New York, NY, USA}
}

@article{cuntzsupersolvablesimplicial,
	title = {Supersolvable simplicial arrangements},
	journal = {Advances in Applied Mathematics},
	volume = {107},
	pages = {32-73},
	year = {2019},
	issn = {0196-8858},
	doi = {10.1016/j.aam.2019.02.008},
	author = {Michael Cuntz and Paul Mücksch}
}

@article{FOLKMAN1978199,
	title = {Oriented matroids},
	journal = {Journal of Combinatorial Theory, Series B},
	volume = {25},
	number = {2},
	pages = {199-236},
	year = {1978},
	issn = {0095-8956},
	doi = {10.1016/0095-8956(78)90039-4},
	author = {Jon Folkman and Jim Lawrence}
}

@book{orlik2013arrangements,
	title={Arrangements of hyperplanes},
	author={Orlik, Peter and Terao, Hiroaki},
	volume={300},
	year={2013},
	publisher={Springer Science \& Business Media}
}

@article{coxeter1934discrete,
	title={Discrete groups generated by reflections},
	author={Coxeter, Harold SM},
	journal={Annals of Mathematics},
	volume={35},
	number={3},
	pages={588--621},
	year={1934},
	publisher={JSTOR}
}

@book{grove1996finite,
	title={Finite reflection groups},
	author={Grove, Larry C and Benson, Clark T},
	volume={99},
	year={1996},
	publisher={Springer Science \& Business Media}
}

@misc{brenner2025combinatorial,
	title={Combinatorial generation via permutation languages. VII. Supersolvable hyperplane arrangements},
	author={Brenner, Sofia and Cardinal, Jean and McConville, Thomas and Merino, Arturo and M{\"u}tze, Torsten},
	year={2025},
	eprint={2507.14327},
	archivePrefix={arXiv},
}

@article{inoue2023hamiltonian,
 author = {Inoue, Takato and Yamane, Hiroyuki},
 title = {Hamiltonian cycles for finite {Weyl} groupoids},
 fjournal = {Journal of Algebra and its Applications},
 journal = {J. Algebra Appl.},
 issn = {0219-4988},
 volume = {25},
 number = {7},
 pages = {40},
 note = {Id/No 2650054},
 year = {2026},
 language = {English},
 doi = {10.1142/S0219498826500544},
 zbMATH = {8155808}
}

@article{yamane2021hamilton,
 author = {Yamane, Hiroyuki},
 title = {Hamilton circuits of {Cayley} graphs of {Weyl} groupoids of generalized quantum groups},
 fjournal = {Toyama Mathematical Journal},
 journal = {Toyama Math. J.},
 issn = {1880-6015},
 volume = {43},
 pages = {1--76},
 year = {2022},
 language = {English},
 zbMATH = {7720056},
 Zbl = {1531.05118}
}

@article{CuntzHeckenberger+2015+77+108,
	title = {Finite Weyl groupoids},
	author = {Cuntz, M. and Heckenberger, I.},
	pages = {77--108},
	volume = {2015},
	number = {702},
	journal = {Journal für die reine und angewandte Mathematik (Crelles Journal)},
	doi = {10.1515/crelle-2013-0033},
	year = {2015}
}

@article{kemper2018odd,
  title={The odd--even invariant and Hamiltonian circuits in tope graphs},
  author={Kemper, Yvonne and Lawrence, Jim},
  journal={European Journal of Combinatorics},
  volume={69},
  pages={76--90},
  year={2018},
  publisher={Elsevier}
}

@manual{sagemath,
  Key          = {SageMath},
  Author       = {{The Sage Developers}},
  Title        = {{S}ageMath, the {S}age {M}athematics {S}oftware {S}ystem ({V}ersions 10.3 and 10.4)},
  note         = {{\tt https://www.sagemath.org}},
  Year         = {2025},
}
\end{document}